\documentclass[12pt]{amsart}
\usepackage{amssymb}
\usepackage{amsfonts}
\usepackage{amsmath}
\usepackage{graphicx}%

\usepackage{tikz}

\newtheorem{theorem}{Theorem}

\newtheorem{corollary}{Corollary}

\newtheorem{proposition}{Proposition}
\theoremstyle{definition}
\newtheorem{remark}{Remark}

\def \bf{\mathbf}

\def \a{\alpha}         
\def \b{\beta}           
\def \th{\theta}       

\def\u{{\bf u}}

\newcommand {\q} {\mathbf{q}}

\title{
Classification of  stacked central configurations in $R^3$}
\begin{document}
\maketitle
\markboth{Xiang Yu and Shuqiang Zhu}{Classification of  stacked  CC in $R^3$}
\vspace{-0.5cm}
\author       
\bigskip
\begin{center}
	{Xiang Yu$^1$  and Shuqiang Zhu$^2,$}\\
	\bigskip
	$^1$School of Economic and Mathematics, Southwestern
University of Finance and Economics, Chengdu 611130, P.R. China \\
	$^2$School of Mathematical Sciences, University of Science and Technology of China, Hefei 230026,   P.R. China\\

	yuxiang@swufe.edu.cn, zhus@ustc.edu.cn
\end{center}

\begin{abstract}
	We classify the extensions of $n$-body central configurations to $(n+1)$-body central configurations in $R^3$, in both the collinear case and the non-collinear case.  We completely solve  the two open questions posed by Hampton (Nonlinearity 18: 2299-2304, 2005).  This classification is  related with  study on co-circular and co-spherical  central  configurations. We also obtain a general property of   co-circular central  configurations. 

\end{abstract}

\noindent \textbf{Key Words:} Newtonian n-body problem; stacked central configurations; co-circular configurations; co-spherical configurations; pyramidal central configurations; perverse solutions.

\section{Introduction}\label{Sec:int}

Central configurations are important in the classical N-body problems. They naturally arise in the study of the self-similar solutions, and they are involved in the classification of the topology of integral manifolds \cite{Sma70-2}.  In the  collection of important open problems in celestial mechanics compiled by Albouy-Cabral-Santos \cite{ACS12}, half of the list   is  on central configurations.  Readers are referred to \cite{Alb03, ACS12, AK12, LMS15, Saa80} for introductions, recent advance and open questions. 

The $(n+k)$-body   central configurations  extended from  $n$-body central configurations  by adding  $k$ bodies   are called \emph{stacked central configurations}.  For  instance, the Lagrangian equilateral triangle central configuration is a stacked central configuration.  It is also  well-known that  a \emph{pyramidal central configuration} can be obtained by adding one mass to a co-circular central configuration \cite{Alb03, Fay96, OXZ04}.   Hampton  introduced    stacked central configurations in 2005 \cite{Ham05}. 
 Many other examples of stacked central configurations were constructed, see \cite{ CLP14,  DS15, HS07,  OC12}.  
 Hampton  also raises two questions regarding stacked central configuration   \cite{Ham05}:
 \begin{enumerate}
 	\item In addition to symmetric collinear configurations\footnote{According to Theorem \ref{thm:collinear},  there is no  five-body  collinear central configuration with a subset forming a  four-body central configuration.  },   the square or a regular tetrahedron with a mass at its center and the square pyramidal configuration  are there any five-body central configurations with a subset forming a four-body central configuration?
 	\item  Are there any five-body   non-collinear central configurations all of whose four-body subsets form a central configuration?
 \end{enumerate}
 There are some works devoted to  this two questions. Assuming that a  five-body central configuration is co-planar and non-collinear, in 2013, Fernandes-Mello \cite{FM13} and Alvarez Ram\'{i}rez-Santos-Vidal \cite{ASV13} announced independently that such configuration must be a square with equal masses and one mass at the center of the square.  Though  the paper  \cite{FM13} contains several inspirational observations, some  argument   is problematic. In 2018,  Fernandes-Mello \cite{FM18} fix the proof,   see Remark \ref{rem:fm2}.   However, the two questions remain open. 
 
 In this work, we classify the ways by which  an $n$-body  $(n\ge 2)$ central configuration can extend to an $(n+1)$-body central configuration in $R^3$. With this classification, we solve the two open questions completely.  We also find one general property of co-circular central configurations. It plays a crucial role in our study of the extensions 
 of co-circular central configurations.

 There are two cases. Firstly, the extended $(n+1)$-body central configuration is collinear. In this  case, we show that  extensions happen only for $n= 2$. So the question of collinear extensions has already been answered by  Euler \cite{Eul}.   
 Secondly, the extended $(n+1)$-body central configuration is non-collinear. In this case, it has been proved by Fernandes-Mello \cite{FM13-1} that   it is necessary that 
  the $n$-body central configuration  lies on a  common circle or sphere. Our  approach  is different from theirs. Our results contain not only the necessary conditions, but  sufficient conditions as well.  Thus, we  can classify all the extensions.   
  
   Let $r$ be the radius of the circle (sphere) and $r_0=(\frac { m}{ \lambda})^{\frac{1}{3}}$,  the cubic root of the ratio of total mass $m$  and the multiplier $\lambda$ of the  $n$-body configuration. The co-circular  (co-spherical) central configurations can extend if their mass center equals their geometric center. They can also extend if  $r_0\ge r$ for the co-circular case, and  
    $r_0=r$ for the co-spherical  case. 

Thus,  the measurement of $r$ and $r_0$ of the co-circular and co-spherical central  configurations is important.  We obtain a general result for  the  co-circular  ones, with which  we  prove that $r_0>r$ holds for all four, five, and six-body co-circular central configurations.  Together with the works of Hampton \cite{Ham03} and  Cors-Roberts \cite{CR12} on the co-circular four-body problem, we  find all   the extensions of four-body central configurations to five-body central configurations.  And we  answer  Hampton's two questions completely. 



The paper is organized as follows. In Section \ref{sec:mainthm}, we state the main results. In Section \ref{sec:pf}, we prove the main results. 
In Section \ref{sec:345}, we discuss the extensions of  two, three,  four, and five-body central configurations. In Section \ref{sec:co-s}, we find several examples of co-spherical central configurations whose mass center equals the geometric center.

 \section{ Main Results}\label{sec:mainthm} 
 

 We are interested in the $n$-body  central configurations that can extend to   $(n+1)$-body central configurations by adding one mass.   If there is no confusion raised, the $n$ masses are 
$m_1, ..., m_n$ and the corresponding configuration is  $\q=(\q_1, ..., \q_n)$.  We denote by   $m_0$  the added  mass and  by   $\q_0$ its position.  We denote by $\bar{\q}=  (\q_0, \q_1, ..., \q_n)$ the extended configuration. We will also call the original $n$-body  configuration $\q$  the sub configuration.   We use $r_{ij}$ to denote the distance between any two of the $n+1$ particles, i.e., $ r_{ij}= |\q_i-\q_j|$,    $ 0\le i<j\le n$.  We denote by $m$ the sum of the $n$ masses, and by $\bar m$ the sum of the $n+1$ masses, i.e, 
\[  m=\sum_{i=1}^{n}m_i, \  \bar m=\sum_{i=0}^{n}m_i=m_0 +m.   \]
 We denote by $c$ the mass center  of the $n$-body sub configuration, and $\bar c$ the mass center  of the $(n+1)$-body configuration, i.e, 
 \[  c=\frac{\sum_{i=1}^{n} m_i \q_i }{m},   \ \bar c=\frac{\sum_{i=0}^{n} m_i \q_i }{\bar m}=  \frac{m c+m_0\q_0}{m+m_0}.   \]
 Denote by $U, I$ and $\bar U, \bar I$ the force function and the momentum of inertia of the sub $n$-body system and the  $(n+1)$-body system respectively, i.e.,  
 \begin{align*}
 &U=\sum_{1 \le i<  j \le n}  \frac{m_im_j}{r_{ij}},    \ \   I=\sum_{1 \le i<  j \le n}  \frac{m_im_j}{m}|\q_i-\q_j|^2=\sum_{i=1} ^n m_i|\q_i-c|^2;\\
 &\bar U=\sum_{0 \le i<  j \le n}  \frac{m_im_j}{r_{ij}} =U + \sum_{i=1}^n \frac{m_0m_i}{r_{i0}},  \\
 & \bar I =\sum_{0 \le i<  j \le n}  \frac{m_im_j}{\bar m}|\q_i-\q_j|^2=\sum_{i=0} ^n m_i|\q_i-\bar c|^2. 
 \end{align*} 
  We assume that both  the  $(n+1)$-body configuration  and the  $n$-body sub  configuration are central. That is, the configurations $\bar \q$ and $\q$ satisfy the following two systems simultaneously, 
 \begin{equation} \label{equ:cc0}   
  \nabla  U(\q) + \lambda/2 \nabla  I(\q) =0, \    \nabla  \bar U(\bar \q) + \bar \lambda/2 \nabla  \bar  I( \bar \q)=0,      
   \end{equation} 
where $ \lambda = U / I  $  and  $ \bar \lambda =\bar U / \bar I $.

 Our first result concerns the case that  the $(n+1)$-body central configuration is collinear.
 \begin{theorem}\label{thm:collinear}
 	Assume that $n\ge2$.  Suppose that 	an $n$-body  collinear  central configuration can extend to an $(n+1)$-body collinear  central configuration by adding one mass, then $n= 2$. 
 \end{theorem}

  This reduces study of  collinear extensions to  study of  the well-known 
 three-body collinear  central configurations, which has  been considered  by  Euler \cite{Eul}, see Section \ref{subsec:2-3}. 

In what  follows, we mainly discuss the non-collinear case.  In this case,  Fernandes-Mello \cite{FM13-1}  have showed  that the  $n$-body sub  configuration must lie on a common circle or sphere and the added mass is at the geometric center. Their proof employed the Laura-Andoyer equations. Our approach  is  different from theirs, see Section \ref{sec:pf}.  Our results contain more details, which enables us to  provide a complete classification of the non-collinear extensions. 
We divide our discussion into two cases: $\q_0=c$ and $\q_0\ne c$. 

\begin{theorem}\label{thm:c=0}
	Suppose  that $\q_0=c$.  Then both  the  $(n+1)$-body configuration  and the  $n$-body sub  configuration are central  if and only if  the following two conditions are satisfied
	\begin{itemize}
		\item 	The  $n$-body sub  configuration is central; 
		\item  $|\q_1-\q_0|=|\q_2-\q_0|=... =|\q_n-\q_0| $.
	\end{itemize}
\end{theorem}

\begin{theorem}\label{thm:cn=0}
	Suppose  that $\q_0\ne c$ and that the $(n+1)$-body  configuration  is non-collinear.  Then both  the  $(n+1)$-body configuration  and the  $n$-body sub  configuration are central  if and only if  the following two conditions are satisfied
	\begin{itemize}
		\item 		The  $n$-body sub  configuration is central; 
		\item  $\frac{1}{|\q_i-\q_0|^3}= \frac{ \lambda}{ m}$, $i=1, ..., n$. 
	\end{itemize}
\end{theorem}

\begin{corollary}
	 Suppose that the   $(n+1)$-body non-collinear  configuration  and the   $n$-body sub  configuration are central, then they are still central if we replace  $m_0$  by an arbitrary mass.
\end{corollary}	


We    answer the second question of Hampton \cite{Ham05}, see Section \ref{Sec:int}.  

\begin{proposition}\label{prop:Hamp2}
	In $R^3$, there are only three types of  $(n+1)$-body central configurations all of whose $n$-body subsets form a central configuration, namely, the   three-body Eulerian collinear central configurations,  the  Lagrangian equilateral triangle central configurations 
	and the regular tetrahedron  central configurations. 
\end{proposition}

The sub configuration we are looking for lies on a common circle or sphere.  These central configurations are called \emph{co-circular central configurations}, in the planar case and \emph{co-spherical  central configurations}, in the spatial case.   To make it precise, we use the terminology ``co-spherical configuration'' to indicate that the configuration  is  not planar.  Denote by $r$ the radius of the related circle (sphere). 
 Denote by $r_0$ the cubic root of the ratio of total mass and the multiplier of the sub  central configuration, i.e., 
$$ r_0= (\frac { m}{ \lambda})^{\frac{1}{3}}=  (\frac { m I}{ U})^{\frac{1}{3}} =  (\frac{ \sum_{1 \le i<  j \le n} m_im_j r_{ij}^2 } { \sum_{1 \le i<  j \le n}   m_im_j /r_{ij}  })^{\frac{1}{3}}.     $$ 
An $(n+1)$-body  spatial central configuration of which $n$ points lie in an affine plane is called  a \emph{pyramidal  central configuration}.

\begin{theorem} \label{thm:class}
	In  $R^3$, there are only five ways that an  $n$-body  central  configuration can extend to an $(n+1)$-body  non-collinear central configuration. 
\begin{itemize}
	\item I co-circular to planar:  $n$-body co-circular central  configurations whose mass center coincides with the geometric center,  extend to $(n+1)$-body planar  central configurations by adding $m_0$ at the geometric center;
		\item  II co-circular to planar:  $n$-body co-circular central  configurations whose mass center does not  coincide with the geometric center, but $r=r_0$,  extend to  $(n+1)$-body planar central configurations by adding $m_0$ at the geometric center;
\item   III co-circular to pyramidal: $n$-body co-circular central  configurations whose mass center may or may not coincide with the geometric center, but $r<r_0$, extend to pyramidal central configurations by adding $m_0$ on the orthogonal axis passing through the center of the circle such that $r_{10}=r_0$; 
		
\item    IV co-spherical to spatial: $n$-body co-spherical central  configurations whose mass center coincides with the geometric center,  extend to $(n+1)$-body central configurations by adding $m_0$ at the geometric center;

\item   V co-spherical to spatial: $n$-body co-spherical central  configurations whose mass center does not coincide with the geometric center,  but $r=r_0$, extend to  $(n+1)$-body central configurations by adding $m_0$ at the geometric center. 
\end{itemize}
\end{theorem}

Chenciner \cite{Che01}  asked: Is the regular $n$-gon with equal masses the unique co-circular central configuration that the center of mass equals the geometric center?   This question is listed as Problem 12 in a collection of  open problems in celestial mechanics compiled by Albouy-Cabral-Santos \cite{ACS12}.  We may ask one equivalent question:  
If  an  $n$-body co-circular  central configuration can extend to a co-planar  central configuration by adding 
 one  mass $m_0$ at the mass  center,  
 does the 
$n$-body central configuration have to be the regular $n$-gon with equal masses? 
Until now, the question has only been answered affirmatively  for $n=4$, by Hampton in 2003 \cite{Ham03}. 

\begin{corollary} \label{cor:r_0}
	Suppose  
	that an   $(n+1)$-body non-collinear central  configuration  is obtained from an   $n$-body co-circular  (co-spherical) central  configuration by adding  one mass $m_0$ such that $r_{i0}=r_0, i=1, ..., n$, i.e., by  way II, III, and V of Theorem \ref{thm:class}. 
	Let $\bar r_0=  (\bar m / \bar \lambda)^{\frac{1}{3}}$,  the cubic root of the ratio of total mass and the multiplier of the extended  $(n+1)$-body central configuration.  Then $\bar r_0$ does not depend on the value of $m_0$, and $ \bar r_0=r_0$. 
\end{corollary}

\begin{remark} \label{rem:fm}
	Fernandes-Mello \cite{FM15} announced that if  an  $n$-body  co-circular central configuration can extend to an $(n+1)$-body central configuration by adding one arbitrary mass at the geometric center,  the mass center of the $n$ bodies  must coincide with the geometric center (Lemma 2.3 of \cite{FM15}). That is to say, the extension way II of Theorem \ref{thm:class}, does not exist. This statement is incorrect.   The well-known Lagrangian equilateral triangle central configuration is a counterexample, see Remark \ref{rem:r02} and Section \ref{subsec:2-3}. The  error in their proof is similar to the one described  in Remark \ref{rem:fm2}.  
\end{remark}


   According to Theorem \ref{thm:class},  the measurement of $r$ and $r_0$ of  the co-circular and co-spherical central  configurations is crucial for the classification of stacked central configurations. We obtain a general result for the   co-circular case, with which, we could prove that $r_0> r$ holds for all four, five, and six-body co-circular central configurations.

   Some \textbf{ notations} for the co-circular configurations:  Edges are  line segments connecting two different vertices of a polygon. For a co-circular configuration whose vertices are ordered counterclockwise as   $(\q_1, ..., \q_n)$, the edges $\overline{\q_i\q_j}$ are called \emph{exterior sides} if $|i-j|=1$ or $n-1$, and \emph{diagonals} otherwise.   An edge and a vertex on that edge are called \emph{incident}. 

\begin{theorem}\label{thm:ccc}
Assume that $n\ge 4$.  	For any  $n$-body co-circular central configuration,   all the exterior sides are less than $r_0$.  At  each vertex,   there is at least one incident diagonal  larger than $r_0$. 
\end{theorem}

\begin{remark} 	For $n=2, 3$, there is no diagonal. It is easy to see that $r_0=r_{12}$ for both of the two cases. 
			  For $n\ge4$, there are at least $n/2$ ($n$ even) or $(n+1)/2$ ($n$ odd) diagonals greater than $r_0$.  For $n=4,5$, these results have been proved for a larger set of central configurations, namely, the  four and five-body  planar convex central configurations, by 
  MacMillan-Bartky \cite{MB32}  and Chen-Hsiao \cite{CH18}  respectively.  Generally, for large $n$,   there would be many  diagonals smaller  than  $r_0$, see the examples in Section \ref{subsec:gon}.
\end{remark}

\begin{corollary} \label{cor:semicir}
	Co-circular central configurations can't lie entirely in a semi-circle. 
\end{corollary}

\begin{remark}
	As suggested by Cors-Roberts \cite{CR12}, this fact also follows nicely from the  \emph{perpendicular bisector theorem} \cite{Moe90}. 
\end{remark}

\begin{proposition}\label{cor:4-6ccc}
	For all four, five, and six-body co-circular central configurations,    the radius of the circle containing the bodies  is smaller than $r_0$. 
\end{proposition}

\begin{remark}\label{rem:r02}
	The two ends of a  segment can be placed on a circle with radius equal or greater than half of the segment. Thus, for $n=2$, we have $ r_0/2\le r<\infty$. For $n=3$,  we have $r_0=\sqrt{3}r$.
\end{remark}
	
	We  find all  the extensions of four-body central configurations to five-body central configurations.  This  answers  Hampton's first question  \cite{Ham05}, see Section \ref{Sec:int}.  
\begin{theorem}\label{thm:4-5}
	There are only three types of five-body  central configurations of which  a   four-body sub  configuration is central:  
	\begin{itemize}
		\item the square with equal masses and an arbitrary mass $m_0$ at the center;
		\item any  four-body co-circular  central configuration and an arbitrary mass $m_0$ on the orthogonal axis passing through the center of the circle.  The height of $m_0$ is $h= \sqrt{   r_0^2 -r^2}$;
		\item the regular tetrahedron  with equal masses and an arbitrary mass $m_0$ at the center.	
	\end{itemize}
\end{theorem} 

For the extension of five and more bodies, we have  some partial results. 
\begin{proposition}\label{prop:5+}
	
	\begin{itemize}
		\item If a  five (six)-body  co-circular  central configuration can extend to a co-planar  central configuration by adding 
		one  mass $m_0$ at the geometric center,  the mass center of the co-circular  central configuration must coincide with  the  geometric center. 
		

		\item Any  five and six-body  co-circular  central configurations can  extend to a pyramidal central configuration by adding 
		an arbitrary mass $m_0$ on the orthogonal axis passing through the center of the circle.  The height of $m_0$ is $h= \sqrt{   r_0^2 -r^2}$. 
	\end{itemize}
\end{proposition}

Our classification of the stacked central configurations also enables us to  give  a complete characterization of the pyramidal central configurations. 

\begin{proposition}\label{prop:pcc}
	Let  $\bar \q=(\q_0, \q_1, ..., \q_n)$ be  an $(n+1)$-body  pyramidal  configuration with masses $ m_0, m_1, ..., m_n$, where $\q_0$ is the top vertex which is off the affine plane containing  $m_1, ..., m_n$.  Then the pyramidal configuration is central if and only if the following conditions are satisfied
	\begin{itemize}
		\item The sub configuration  $\q_1, ..., \q_n$  is  central and co-circular; 
		\item $r<r_0$, where $r, r_0$ are values associated with the $n$-body co-circular central configuration;
		\item The top vertex $\q_0$  is on the orthogonal axis passing through the center of the circle.  The height of $\q_0$ is $h= \sqrt{   r_0^2 -r^2}$. 
	\end{itemize}
	The top mass $m_0$ is arbitrary.  
\end{proposition}

\begin{remark}
	When studying five-body pyramidal configurations, Fay\c{c}al \cite{Fay96} showed that the base must be co-circular and that the top mass is arbitrary. She  also gave a formula for the distance between the top vertex and the base vertices.   Ouyang-Xie-Zhang \cite{OXZ04} generalized Fay\c{c}al's result to   $n$-body pyramidal configurations. 
	Their characterization is almost the same as ours. They did not    compare explicitly the two values $r$ and $r_0$.   Albouy  \cite{Alb03} obtained a general proposition on central configurations in $R^N$.  Restricted in $R^3$,  it  immediately implies the base of any pyramidal central configuration is central and co-circular.    
\end{remark}

Stacked central configurations are also related with \emph{perverse solutions} introduced by Chenciner \cite{Che01}.  
A solution $\q(t)=\q_1(t), ..., \q_n(t)$  of the $n$-body problem with masses $m_1, m_2, ...,  m_n$ is called a perverse  solution if there  exists another system of masses, $m'_1, m'_2,$ $ ...,  m'_n$, for which $\q(t)$ is still a solution.  Note that any 
$(n+1)$-body non-collinear  central configuration obtained  from an $n$-body  central configuration by adding one mass at the mass center of the $n$ bodies,  i.e., by  way I and IV  of Theorem \ref{thm:class},  would provide perverse solutions, namely, the relative equilibrium and the total collision solution for the planar case, and the  total collision solution for the spatial case, see also Section \ref{subsec:dron}. 


\section{Proofs of the Main Results} \label{sec:pf}
We first simplify the central configuration equations \eqref{equ:cc0}. 

Note that there is  a simple but important fact:  the three points $c, \bar c,$ and $\q_0$ are collinear.  In fact,  $\bar c$ can also be seen as the mass center of the two material points  $c, \q_0$ with masses $m, m_0$.  Thus,  $\bar c$ equals $\q_0$ if and only if  $ c$ equals $\q_0$.  The collinearity  of the three points is also revealed in the following equalities 
\begin{equation} \label{equ:mc}
\bar m (\bar c-\q_0 ) =  m ( c-\q_0 ), \  A-\bar c = (A-  c)- \frac{m_0}{\bar m} (\q_0 -c), \ 
\end{equation}
where $A$ is an arbitrary point.

We assume that both  the  $(n+1)$-body configuration  and the  $n$-body sub  configuration are central.  The central configuration equations \eqref{equ:cc0} can be written as 
\begin{equation} \label{equ:cc1}
\begin{cases}
&\sum_{j\ne i, j=1}^n m_im_j \frac{\q_j-\q_i}{|\q_j-\q_i|^3}= - \lambda m_i(\q_i-c),  i=1, ..., n,   \cr
&      \sum_{i=1}^n \frac{m_im_0(\q_i-\q_0)}{r_{i0}^3} = - \bar \lambda m_0(\q_0-\bar c),                           \cr
&  \sum_{j\ne i, j=1}^n m_im_j \frac{\q_j-\q_i}{|\q_j-\q_i|^3} + \frac{m_im_0(\q_0-\q_i)}{r_{i0}^3} = - \bar \lambda m_i(\q_i-\bar c), \ i= 1, ..., n, 
\end{cases}
\end{equation}
where $ \lambda = U / I  $  and  $ \bar \lambda =\bar U / \bar I $. 

Note that the third  part  of system \eqref{equ:cc1}  can be written as 
\begin{equation} \label{equ:form1}
\lambda m_i(\q_i-c)+   \frac{m_0m_i(\q_i-\q_0)}{r_{i0}^3} =  \bar \lambda m_i(\q_i-\bar c), \ i= 1, ..., n.
\end{equation} Summing up all the $n$ equations gives the second part of system \eqref{equ:cc1}.  Furthermore,  
by \eqref{equ:mc}, the $n$ equations are equivalent to 
\begin{equation}\label{equ:form2}
( \lambda- \bar \lambda + \frac{m_0}{r_{i0}^3}) (\q_i-c)=  ( \frac{m_0}{r_{i0}^3} -  \frac{\bar \lambda m_0}{\bar m} ) (\q_0- c), \ \ i= 1, ..., n.
\end{equation}
Therefore, system \eqref{equ:cc0}     is equivalent to the following system,
\begin{equation} \label{equ:cc3}
\begin{cases}
&\sum_{j\ne i, j=1}^n m_im_j \frac{\q_j-\q_i}{|\q_j-\q_i|^3}= - \lambda m_i(\q_i-c),  i=1, ..., n,   \cr
&( \lambda- \bar \lambda + \frac{m_0}{r_{i0}^3}) (\q_i-c)=  ( \frac{m_0}{r_{i0}^3} -  \frac{\bar \lambda m_0}{\bar m} ) (\q_0- c), \ i= 1, ..., n, 
\end{cases}
\end{equation}
where $ \lambda = U / I  $  and  $ \bar \lambda =\bar U / \bar I $.

\begin{proof} [Proof of Theorem \ref{thm:collinear}]
	Note that 
	equations \eqref{equ:form2} can be written as
	\begin{equation}\notag 
	( \lambda- \bar \lambda) (\q_i-\q_0)  + \frac{m_0(\q_i-\q_0)}{r_{i0}^3} +( \lambda- \frac{\bar \lambda}{\bar m}m )( \q_0-c)  =0, \ i= 1, ..., n.
	\end{equation}
	Assume that all the particles are on the $x$-axis, and use $x_i$ to denote the position of $m_i$.  
	The equations become 
	\begin{equation}\label{equ:ecc2}
	( \lambda- \bar \lambda) (x_i-x_0)  + \frac{m_0(x_i-x_0)}{|x_i-x_0|^3} +( \lambda- \frac{\bar \lambda}{\bar m}m )( x_0-c)  =0, \ i= 1, ..., n.
	\end{equation}
	Let $y_i=x_i-x_0$,  $\a=( \lambda- \frac{\bar \lambda}{\bar m}m )( x_0-c)$, and $\b=\lambda- \bar \lambda$.  The above equation  implies that each of  the $n$ distinct real nonzero  values $\{y_1, ..., y_n\}$ satisfies a common algebraic equation, 
	\begin{equation*}
	\b  y+ \frac{m_0y}{|y|^3} +\a =0. 
	\end{equation*}
	Assume that there are $k$ particles on the left side of $m_0$ and $n-k$ particles on the right, $0\le k\le n$. 
	Then 	the cubic equation 
	\begin{equation}\label{equ:ecc3}
	-m_0z^3  + \a z+\b  =0,
	\end{equation} has at least  $k$ distinct negative roots, and  the cubic equation 
	\begin{equation}\label{equ:ecc4}
	m_0z^3  + \a z+\b  =0
	\end{equation}
	has at least  $n-k$ distinct positive roots. 

	Assume that $z_1, z_2, z_3$ are the roots of   equation \eqref{equ:ecc3} and $z_4, z_5, z_6$ are the roots of  equation \eqref{equ:ecc4}.  
	We are going to finish the proof by showing  that the sum of the number of  negative roots of equation \eqref{equ:ecc3} and the  number of  positive  roots of equation \eqref{equ:ecc4} is not greater than $2$. This is obviously true if $\a=0$ or $\b=0$. So we assume that $\a\ne 0$ and  $\b\ne0$. 

	Recall that a generic cubic equation  $ az^{3}+bz^{2}+cz+d=0,  \  a\ne 0$
	has  only one real root  and two conjugate imaginary roots if and only if 
	\[ \Delta =18abcd-4b^{3}d+b^{2}c^{2}-4ac^{3}-27a^{2}d^{2} <0. \]
	For our two cubic equations, we have 
	\[  \Delta_1=4m_0\a^3 -27m_0^2\b^2, \   \Delta_2=-4m_0\a^3 -27m_0^2\b^2.    \]
	Obviously, at least one of $\Delta_1$ and  $\Delta_2$  is negative. Without loss of generality, we assume that $\Delta_1 < 0$, then only one of $z_1, z_2, z_3$ is real, say, $z_1$,  and the other two  are  conjugate imaginary numbers.
	
	By Vieta's formulas,  
	$ z_4+z_5+z_6= 0$, which implies that at most two of the roots of equation \eqref{equ:ecc4}  are positive. Thus, we  assume that $z_1<0$.  By Vieta's formulas, 
	$z_4z_5z_6=-z_1z_2z_3>0, z_4+z_5+z_6= 0$, 
	which implies that equation \eqref{equ:ecc4} has only one positive root.  In words, 	the sum of  the number of  negative roots of equation \eqref{equ:ecc3} and the  number of  positive  roots of equation \eqref{equ:ecc4} is not greater than $2$.  This completes the proof. 

\end{proof}

\begin{proof} [Proof of Theorem \ref{thm:c=0}]
	The proof of the necessary conditions:  First, 	the  $n$-body sub  configuration must be  central.   In this case,  we have   $\q_0=c=\bar c$, so   $\q_i\ne c, \ i=1, ..., n$.   Then the second part of  system \eqref{equ:cc3} implies that  $\frac{1}{|\q_i-\q_0|^3}=\frac{\bar \lambda -\lambda}{m_0}$ for  $i=1, ..., n$.  So we obtain that $|\q_1-\q_0|=|\q_2-\q_0|=... =|\q_n-\q_0| $.
	
	The proof of the sufficient  conditions:   The first part  of system  \eqref{equ:cc3}  obviously  holds.   Since $\q_0=c=\bar c$ and $r_{10}=...=r_{n0}$,    we have 
	\begin{align*}
	&\bar I=  \sum_{i=0} ^n m_i|\q_i-\bar c|^2=\sum_{i=1} ^n m_i|\q_i-c|^2 =I= m r_{10}^2,\\
	&\bar \lambda = \frac{U + \sum_{i=1}^n \frac{m_im_0}{r_{i0}}  }{\bar I }  =\lambda +  m_0 \frac{ m }{ I r_{10}}  = \lambda +  m_0 \frac{ 1 }{ r_{10}^3}, 
	\end{align*}
	which   implies that $\frac{1}{|\q_i-\q_0|^3}=\frac{\bar \lambda -\lambda}{m_0}$ for $i=1, ..., n$. Thus,  the second part of  system  \eqref{equ:cc3} holds and the proof is completed. 
\end{proof}

\begin{proof}[Proof of Theorem \ref{thm:cn=0}]
	The proof of the necessary conditions:  First, 	the  $n$-body sub  configuration must be  central.   There exists some body  not on the line $\overline{c \q_0}$ since the configuration $\bar{\q}$ is non-collinear. Suppose that $\q_k\notin \overline{c \q_0}$, then the $k$-th equation  of the second part  of system \eqref{equ:cc3} holds only if 
	\[  \frac{1}{|\q_k-\q_0|^3}=\frac{\bar \lambda -\lambda}{m_0}= \frac{\bar \lambda}{\bar m}.   \]
	Note that $|\q_k-\q_0|=|\q_j-\q_0|$ also holds if  $\q_k\notin \overline{c \q_0},  \q_j\in \overline{c \q_0}$. 
	By \eqref{equ:form1}, we have 
	\[	\lambda (\q_i-c)+   \frac{m_0(\q_i-\q_0)}{r_{i0}^3} =  \bar \lambda (\q_i-\bar c), \ i= k, j.\]
	Subtracting the two equations, we obtain	
	\begin{align*}
	&  \lambda (\q_k-\q_j) + \frac{m_0(\q_k-\q_0)}{r_{k0}^3} -\frac{m_0(\q_j-\q_0)}{r_{j0}^3} =  \bar \lambda (\q_k-\q_j), \\
	&  \frac{(\q_k-\q_0)}{r_{k0}^3} -\frac{(\q_j-\q_0)}{r_{j0}^3} =  \frac{\bar \lambda -\lambda}{m_0}(\q_k-\q_j)= \frac{(\q_k-\q_0)-(\q_j-\q_0)}{r_{k0}^3}, 
	\end{align*}
	which implies that $|\q_k-\q_0|=|\q_j-\q_0|$. Thus,  we obtain 
	\[  \frac{1}{|\q_i-\q_0|^3}=\frac{\bar \lambda -\lambda}{m_0}= \frac{\bar \lambda}{\bar m}, \  \  i=1, ..., n.  \]
	The equality $\frac{\bar \lambda -\lambda}{m_0}= \frac{\bar \lambda}{\bar m}$ implies that $\frac{ \lambda}{ m}= \frac{\bar \lambda}{\bar m}$, so we have $\frac{1}{|\q_i-\q_0|^3}= \frac{ \lambda}{ m}$ for  $i=1, ..., n$.

	The proof of the sufficient  conditions:   The first part of system  \eqref{equ:cc3}  obviously  holds.  With the condition   $\frac{1}{|\q_i-\q_0|^3}= \frac{ \lambda}{ m},   i=1, ..., n$, we obtain 
	\begin{equation*} 
	\begin{split}
	\frac{\bar \lambda}{\bar m}&=   \bar U /(\bar m  \bar I) = \frac{U + \sum_{i=1}^n \frac{m_0m_i}{r_{i0}}}{\sum_{0 \le i<  j \le n}  m_im_jr_{ij}^2} 
	= \frac{U + \sum_{i=1}^n \frac{m_0m_i}{r_{10}}   }{   mI + \sum_{i=1} ^n m_0m_ir_{i0}^2      } \\
	&=  \frac{\lambda I  + \frac{m_0 \lambda r_{10}^3}{r_{10}}   }{   mI +  m_0m r_{10}^2      } 
	=\frac{ \lambda}{ m}. 
	\end{split}
	\end{equation*}	
	The equality  $\frac{ \lambda}{ m}= \frac{\bar \lambda}{\bar m}$   implies that $ \frac{ \lambda}{ m} = \frac{\bar \lambda -\lambda}{m_0}= \frac{\bar \lambda}{\bar m}$,  so we have 
	$$\frac{1}{|\q_i-\q_0|^3}=\frac{\bar \lambda -\lambda}{m_0}= \frac{\bar \lambda}{\bar m}, \ \  i=1, ..., n.$$   
	Thus,  the second part of  system  \eqref{equ:cc3} holds and the proof is completed. 
\end{proof}

\begin{proof}[Proof of Proposition \ref{prop:Hamp2}]
	In the collinear case,  Theorem \ref{thm:collinear} implies that it is possible if and only if $n=2$.  For the non-collinear case, 	
	 Theorem \ref{thm:c=0} and Theorem \ref{thm:cn=0} implies that 
$r_{01}=...=r_{0n}=r_{12}...=r_{1n} ...= r_{n-1,n}, $
	which happens only  if  the $n+1$ bodies   form a regular polytope.  In $R^3$,  that is, the  equilateral triangle  and the regular tetrahedron.  On the other hand,  it is well-known that these configurations  with arbitrary masses are central. This completes the proof. 
\end{proof}

\begin{proof}[Proof of Theorem \ref{thm:class}]
	It is clear from  Theorem \ref{thm:c=0} and Theorem \ref{thm:cn=0}.
\end{proof}

\begin{proof}[Proof of Corollary \ref{cor:r_0}]
	It is clear from the proof of Theorem \ref{thm:cn=0}.
\end{proof}

\begin{proof}[Proof of Theorem \ref{thm:ccc}]
	Assume that $\q_i=(\cos \th_i, \sin \th_i)$ and $0=\th_1< \th_2 <...< \th_n<2\pi$.  The configuration is  central if and only if    
	\begin{equation}\label{equ:ccce}
	\begin{split}
	\sum_{k\ne j}     m_k ( \frac{1}{r^3_{kj}} -  \frac{\lambda}{m }) (\q_k -\q_j )
	&= \sum_{k\ne j}     m_k S_{kj} (\q_k -\q_j )=0,
	\end{split}
	\end{equation}
	for $ j=1, ..., n, $  where $S_{kj}=	\frac{1}{r^3_{kj}} -  \frac{1}{r_0^3}$. 
	
	
	We first show that the two exterior sides incident with  $\q_1$ are smaller than $r_0$ by contradiction.  Note that the sequence $\{ r_{12}, r_{13}, ..., r_{1n} \}$ is either  monotonic or at first increasing  then decreasing.   
	
	Case I:  $ r_{12}\ge r_0, \  r_{n1} \ge r_0$. Then we have 
	\[ r _{k1}> \min\{r_{12}, r_{n1}\}\ge r_0,  \   \frac{1}{r^3_{k1}} - \frac{1}{r_0^3} <0, \    k= 3, ..., n-1.   \]
	That is, $S_{k1}\le 0$ for $k=2, ..., n$. 	
	Denote by $l$ the line perpendicular with the tangent  of the circle at $\q_1$ ( the dashed line), see Figure 1, left, and by $P_l \u$ the orthogonal projection of vector $\u$ along  the line $l$. Then it  is easy to see  that 
	\[  P_l \left(  \sum_{k\ne 1}     m_k S_{k1}  (\q_k -\q_1 )\right)=   \sum_{k\ne 1}      m_k S_{k1}  P_l  (\q_k -\q_1 )\ne 0.  \]
	Therefore,   the first equation of system \eqref{equ:ccce}, $0=  \sum_{k\ne 1}     m_k S_{k1}  (\q_k -\q_1 )$,  can not hold. This is a  contradiction.

	\begin{figure} [!h] \label{fig:c.cc1}
		\begin{center}
			\includegraphics [width=1 \textwidth] {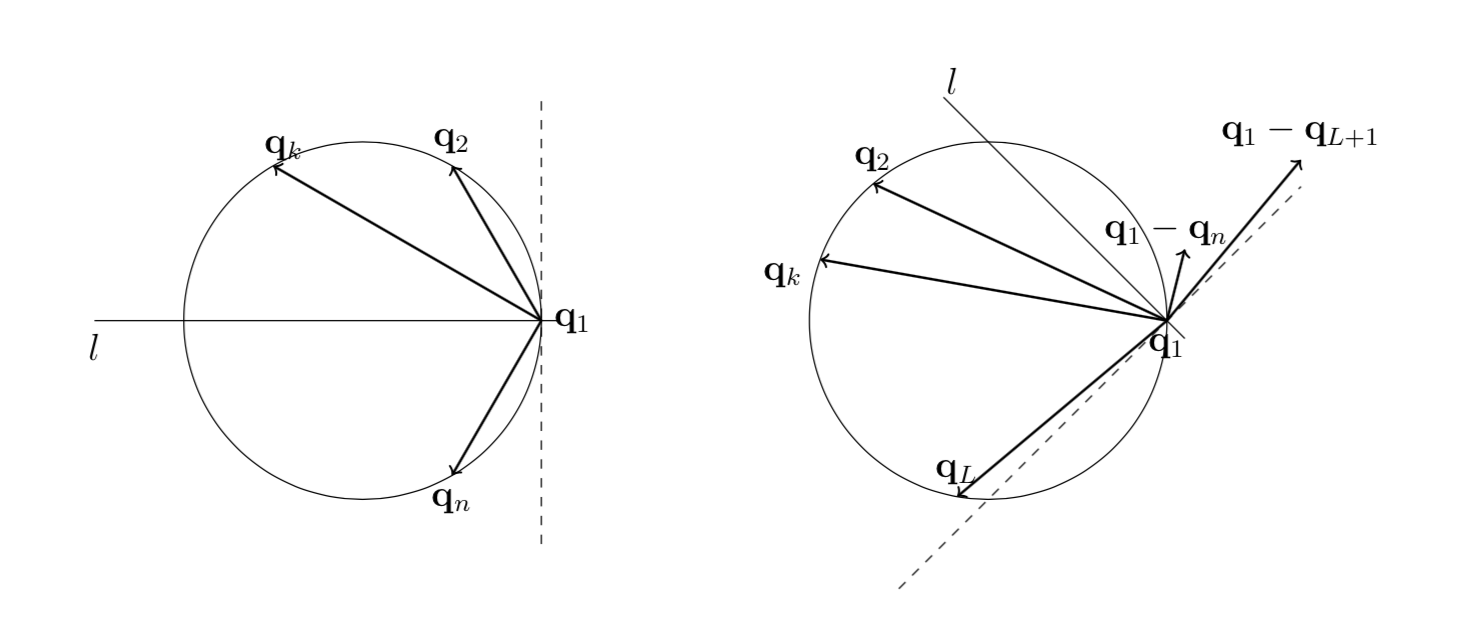}
		\end{center}
		\caption{Co-circular configuration with $ r_{12}\ge r_0, \  r_{n1} \ge r_0$, left,  and  $ r_{12}\ge r_0, \  r_{n1} < r_0$, right.  }
	\end{figure}

	Case II: Only one of the incident exterior side is smaller than $r_0$, say, $ r_{12}\ge r_0, \  r_{n1} < r_0$.  Suppose that $\th_2 < \th_k< ..< \th_L<\th_{L+1}<...< \th_n <2\pi$ and that 
	\[ r_{1L}\ge r_0, \  r_{1, L+1} <r_0.   \] 
	That is, $S_{k1}\le  0$ for $k=2, ..., L$, and  	$S_{k1}>0$ for $k=L+1, ..., n$.  Connect $\q_1$ and one point   between $\q_L$ and $\q_{L+1}$ on the circle by the dashed line, and 	denote by $l$ the line perpendicular with the dashed line, see Figure 1, right. 
	Then it  is easy to see  that 
	\begin{equation*} 
	\begin{split} 
	& P_l \left(  \sum_{k= 2}^L    m_k S_{k1}  (\q_k -\q_1 ) +\sum_{k= L+1}^n    -m_k S_{k1}  (\q_1 -\q_k ) \right)\ne 0.  	\end{split}
	\end{equation*}
	Therefore,   the first equation of system \eqref{equ:ccce}, $0=  \sum_{k\ne 1}     m_k S_{k1}  (\q_k -\q_1 )$,  can not hold. This is a  contradiction.
	
	
	
	We conclude that the two exterior sides incident with  $\q_1$ are smaller than $r_0$, i.e., $S_{12}>0, S_{1n}>0$. If the values  $S_{12}, S_{13}, ..., S_{1n}$  are  all positive, the equation $0=  \sum_{k\ne 1}     m_k S_{k1}  (\q_k -\q_1 )$ can not hold neither, see 
	Figure $1$, left. Thus, there is at least a negative one that must correspond to a diagonal. 
	Hence, there is at least one diagonal incident with $\q_1$ that is larger than $r_0$. 
	
	By symmetry, the statement made for the edges incident with $\q_1$ also holds for the edges incident with any other vertex.  Therefore, we have proved that all the exterior sides are less than $r_0$ and that  there is at least one incident diagonal  larger than $r_0$ at each vertex. 
\end{proof}

\begin{proof}[Proof of Corollary \ref{cor:semicir}]
	If a co-circular configuration lies entirely in a semi-circle, then there is one exterior side longer than all the diagonals. By  Theorem \ref{thm:ccc}, it is not central.  
\end{proof}

\begin{proof}	[Proof of Corollary \ref{cor:4-6ccc}]
	We only  prove for  the six-body  case.  The other cases are similar. 
	Order the six masses  sequentially on the circle as in Figure 2.    First note that  the center of the circle, $O$, must be in the convex hull of the six masses since the  masses are not in a semi-circle.  Assume that $r\ge r_0$. Then 
	Theorem  \ref{thm:ccc}  implies that 
	\[ r_{12}, r_{23}, r_{34}, r_{45}, r_{56}, r_{61}  < r.    \]
			\begin{figure} [h!] \label{fig:c.cc}
		\begin{center}
			\includegraphics [width=0.5 \textwidth] {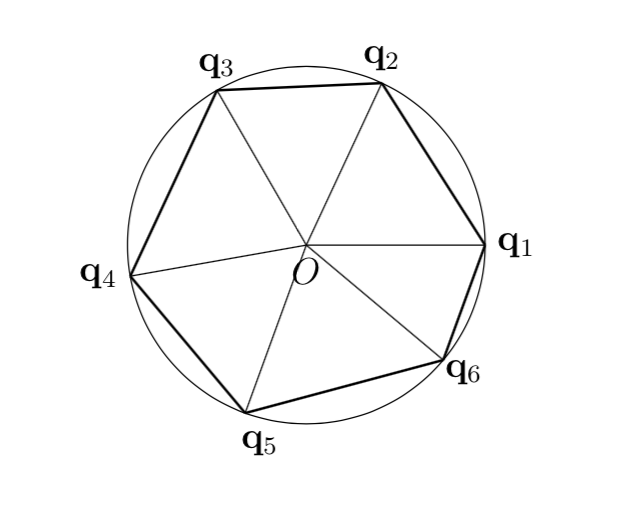}
		\end{center}
		\caption{An example of a co-circular configuration. The center of the circumscribing circle is marked with $O$. }
	\end{figure}
	This  implies that  each of the  six angles $\angle \q_1O \q_2, \angle \q_2O \q_3, \angle \q_3O \q_4,$  $\angle \q_4O \q_5$,  $\angle \q_5O \q_6$, $\angle \q_6O \q_1$,  is strictly  less than $\pi/3$. It is a contradiction since the sum has to be  $2\pi$.

	
\end{proof}

We postpone the proof of  Theorem \ref{thm:4-5} and Proposition \ref{prop:5+}   to Section \ref{sec:345}.

\begin{proof}[Proof of Proposition \ref{prop:pcc}]
	The proof of the necessary conditions:  By Albouy \cite{Alb03}, Ouyang-Xie-Zhang \cite{OXZ04}, the sub configuration $\q$ must be central. 	Thus, both the pyramidal configuration and the sub configuration are central  and $\q_0\ne c$. Then the other conditions follows easily from  Theorem \ref{thm:cn=0}. 
	
	The proof of the sufficient conditions:  By Theorem \ref{thm:cn=0}, if these conditions are satisfied, the pyramidal configuration $\bar \q$ must be central. 
\end{proof}

\section{The extensions of  two, three,  four,  and more  body central configurations}  \label{sec:345}

In this section, we discuss  the extensions of  $n$-body   central configurations to $(n+1)$-body   central configurations 
 for small $n$.  If $n\le 4$, we understand thoroughly  the extensions. If $n\ge 5$, we can only get some partial results. 

\subsection{Two bodies to three} \label{subsec:2-3}   There is only one two-body central configuration, 
namely,  a segment with two arbitrary masses at the ends.  It  is obviously co-circular and the circumscribed circle  is not unique. The radius is in the range $[\frac{r_{12}}{2},  \infty)$. 
\begin{itemize}
	\item I  co-circular to planar: It is easy to see that the mass center coincides the geometric center if and only if the two masses are equal. In this case, we could extend it by adding an arbitrary mass $m_0$ at the center, which is a  three-body collinear central configuration.
	
	\item II co-circular to planar:    Since $r_0=r_{12}$,  the range of radius of the circumscribed circle is $[\frac{r_0}{2},  \infty)$. 
	  It is easy to see that we can extend it by adding  one arbitrary mass $m_0$ on the orthogonal bisector of $\overline{\q_1\q_2}$ such that  $r_{01}=r_{02}=r=r_0=r_{12}$.  The three masses are all arbitrary and the triangle is equilateral. In other words, we have provided another proof of the well-known fact that  the  equilateral triangle with  three arbitrary masses is a central configuration \cite{Lag}. 

		\item III co-circular to pyramidal: Not exist. 
\item IV and V: 	Not exist. 	
		 \end{itemize}

The two-body central configurations  can also  extend to other  three-body collinear central configurations. 
Assume that the  central configuration is on the $x$-axis,  with positions $x_1, x_2$,  $x_1<x_2$.  For any given mass $m_0$, it is easy to show that there is a unique position $x_0$ in each of the three intervals,  $(-\infty, x_1), (x_1, x_2), $ and $ (x_2, \infty)$,  such that the configuration $(x_1, x_2, x_3)$ is  central,   which is the  well-known  three-body  collinear Eulerian central configurations \cite{Eul}.


For  $n\ge 3$, by Theorem \ref{thm:collinear},  any $n$-body  central configuration can not extend to an $(n+1)$-body collinear  central configuration. And that 
  an $n$-body collinear configuration can not extend to an  $(n+1)$-body non-collinear central  configuration by the \emph{perpendicular bisector theorem} \cite{Moe90}. Therefore, we only need to discuss extension of the  $n$-body non-collinear  central configurations in the following cases. 
\subsection{Three bodies to four} This has been considered by Hampton \cite{Ham05}. 
In the three-body case, the only non-collinear  central configuration is the equilateral triangle with three arbitrary masses, which is co-circular.  It is easy to see that $r_0=r_{12}=r_{13}=r_{23}=\sqrt{3}r$.  
\begin{itemize}
	\item I  co-circular to planar: It is easy to see that the mass center coincides the geometric center if and only if the three masses are equal. In this case, we could extend it by adding an arbitrary mass $m_0$ at the center. 
	
	\item II co-circular to planar:  Not exist, since  $r_0>r$. 
	
	\item III co-circular to pyramidal:   As mentioned above, $r_0=\sqrt{3}r$ holds  for any equilateral triangle central configuration. Thus, any equilateral triangle central configuration can extend to a pyramidal central configuration by adding one arbitrary mass $m_0$ such that $r_{10}=r_{20}=r_{30}=r_0=r_{12}$.  In other words, we have provided another proof of the well-known fact that  the  regular tetrahedron  with  four arbitrary masses is a central configuration. 
	\item IV and V: 	Not exist. 	
\end{itemize}

\subsection{Four bodies to five} 
In the four-body case, the only spatial  central configuration is the regular tetrahedron with arbitrary masses, which is co-spherical. On the other hand,  the co-circular central configurations are very rich, and it has been 
studied thoroughly by  Cors-Roberts \cite{CR12}.  
\begin{itemize}
	\item I co-circular to planar:  It has been proved by Hampton \cite{Ham03} that there is only one  four-body co-circular central configuration with mass center at the geometric center, namely, the square with equal masses.  In this case, we could extend it by adding an arbitrary mass $m_0$ at the center. 
	
	\item II co-circular to planar:  Not exist, since that $r_0>r$ for any  four-body co-circular central configuration by Corollary \ref{cor:4-6ccc}. 	
	
	\item III co-circular to pyramidal:  
	Any four-body co-circular central configuration \cite{CR12} could extend to a five-body  pyramidal central configuration. 

	\item IV co-spherical to spatial:   It is easy to see that the mass center of the  regular tetrahedron central configuration coincides the geometric center if and only if the four masses are equal. In this case, we could extend it by adding an arbitrary mass $m_0$ at the center. 
	
	\item V co-spherical to spatial:   Not exist, since that $r_0=r_{12}=\frac{2\sqrt{6}}{3}r$ for any regular tetrahedron central configuration. 			
\end{itemize}

This discussion proves Theorem \ref{thm:4-5}.

\begin{remark}\label{rem:fm2}
	Fernandes-Mello in 2013 \cite{FM13} also announced  that a   four-body co-circular central configuration  can be extended to a  five-body co-planar central configuration if and only if  the configuration is a square with equal masses. 
	However, their original proof is incorrect. On page 302 of \cite{FM13}, 
 where the authors claim that 
	the equation $r^3=\bar m /\bar \lambda $, (with our  notations),  leads to a quadratic polynomial in $m_0$. But from the proof of Theorem \ref{thm:cn=0}, we see that if $r^3=  m/\lambda$, in which no $m_0$ is involved, then $r^3=\bar m /\bar \lambda $ is just an identity for any $m_0$. 	Chen-Hsiao  pointed out this error in 2018 \cite{CH18}. 
	After a preliminary  vision of this paper  was completed,  we were informed that Fernandes-Mello  have corrected their proof in 2018 \cite{FM18}. 
\end{remark}	

  Cors-Roberts \cite{CR12} showed that the  four-body  co-circular  central configurations form a two-dimensional surface. Thus, the   five-body pyramidal central configurations  also  form
  a two-dimensional surface. 
  The property of the  five-body pyramidal central configurations are really rich.  
 We  state some properties about them. They are straightforward corollaries of the results in \cite{CR12}. 
 
 \begin{proposition}
  Not all choices of five  positive masses lead to a five-body pyramidal central configuration. 
 \end{proposition}

 \begin{proposition}
 For a  five-body pyramidal central configuration, let $m_1, m_2, m_3, m_4$ be the four masses of  the co-circular  base.    If just two of the four masses are equal, then the base  configuration is symmetric, either a kite or an isosceles trapezoid. 
 If any three of the four masses are equal, then the base configuration is a square and all four masses are necessarily equal.
\end{proposition}

\subsection{Five and more bodies} 
In the five and more body case, both the co-circular and co-spherical  central configurations are rich,  but   much less  research has been done in this direction. We only state some results known to us.  
 
\begin{itemize}
	\item I co-circular to planar:  Obviously, the regular $n$-gon  $(n\ge 5)$ with equal masses are examples.   We could extend them by adding an arbitrary mass $m_0$ at the center. However, until  now,  we do not know that whether there exist other examples or not,  since the question of Chenciner remains unsolved for $n\ge5$, see the comment  after Theorem \ref{thm:class}. 
	
	
	\item II co-circular to planar:  Not exist for five and six-body case, since that $r_0>r$ in  these cases 
	by Corollary \ref{cor:4-6ccc}.  For more bodies, we have not  found any such example yet. 	
	
	\item III co-circular to pyramidal:  
	Any five and six-body  co-circular central configuration  could extend to a six and seven-body pyramidal central configuration. 
	 For more bodies, we have no general results, see Subsection \ref{subsec:gon}.

	\item IV co-spherical to spatial:   For the five to ten-body cases, We  have some examples of  co-spherical central configurations whose mass center equals the geometric center, see Section \ref{subsec:dron}.   In those cases, we could extend it by adding an arbitrary mass $m_0$ at the center. 
	
	\item V co-spherical to spatial:   For the five and more body case, We do not know that whether there  exist  co-spherical central configurations with $r=r_0$ or not. 
\end{itemize}
This discussion proves Proposition \ref{prop:5+}.



\section{Regular polygons and some examples of  co-spherical central configurations} \label{sec:co-s}

In this section,  we discuss  the regular  polygonal central configurations and construct some   co-spherical central configurations.  Some of them  have mass center at the sphere center. Thus, they can  extend by adding one mass at the center.

\subsection{ Regular polygons }\label{subsec:gon}
 Consider the    regular $n$-gon with equal masses.   Obviously,  they can extend to planar central configurations by adding one mass at the center.   Whether   they can   extend to pyramidal central configurations depends on the measurement of   $r_0$ and $r$.   The  following result  was first showed  by Ouyang-Xie-Zhang \cite{OXZ04}. 
 
 \begin{proposition}
 	The regular $n$-gon with equal  masses can extend to pyramidal central configurations 
 	if and only if $n\le 472$. 
 \end{proposition}
\begin{proof}

   Assume that $m_k=1$,  the radius of the circle is  $1$, and that  the  positions are 
 $\q_k=  e^{ \sqrt{-1} \th_k}, \  \   \th_k= \frac{2k\pi}{n}, \ k=1, ..., n. $
 Then we have 
 \begin{equation*} 
 \begin{split}
 r_0^3&=   \frac{ \sum_{j, k=1} ^n  m_jm_k r_{jk}^2 } { \sum_{j,k=1, j\ne k} ^n  m_jm_k /r_{jk}  } =    \frac{ \sum_{ k=1} ^{n-1} |1-\rho_k|^2 } {\sum_{ k=1} ^{n-1} \frac{1}{|1-\rho_k|} }=\frac{n}{ A(n)},
 \end{split}
 \end{equation*}
 where $A(n)=\frac{1}{2}\sum_{ k=1} ^{n-1} \frac{1}{|1-\rho_k|}= \frac{1}{4}\sum_{ k=1} ^{n-1} \csc(\frac{k\pi}{n})$.
 
 It has been found by  Moeckel-Sim\'o  \cite{MS95} that $\frac{ n}{A(n)}$  is decreasing, and  that the asymptotic expansion of $\frac{A(n)}{n}$ for $n$ large  is 
 		$$\frac{A(n)}{n}\sim 
 		\frac{1}{2\pi}\left(\gamma + \log \frac{2n}{\pi}  \right) + \sum_{k\ge1} \frac{(-1)^k(2^{2k-1}-1)B^2_{2k}\pi^{2k-1}}{(2k)(2k)!n^{2k}},    $$
 where $\gamma$ and $B_{2k}$ stand for the Euler-Mascheroni constant and the Bernoulli numbers respectively.

 	Computation by Matlab shows
 \[  \frac{ 472}{A(472)} \approx 
 1.001, \ \   {\rm while} \ \ 
 \frac{ 473}{A(473)} \approx 
 0.9998.   \]
 So $r_0=(\frac{n}{ A(n)})^\frac{1}{3}  > r$ if and only if $n\le 472. $  This  completes the proof. 
	\end{proof}

\subsection{Co-spherical central configurations}\label{subsec:dron}

There are much less research on  co-spherical central configurations, compared with the co-circular ones. 
Corbera-Llibre-P\'{e}rez \cite{CLP14} constructed three families of central configurations,  each consisting  of  a regular polyhedron  and its dual. In each family, there is a co-spherical one, and the  mass center equals its geometric center.  

We construct some co-spherical central configurations  related with  some co-circular ones.  Let us introduce some \textbf{notations} that will be used only in this subsection. For co-circular central configurations,  we denote by $r, r_0$ the radius of the circumscribing circle and  the cubic root of the ratio of total mass and the multiplier respectively.  These planar configurations will extend to co-spherical central configurations. For the co-spherical ones,  we denote by $R, R_0$ the radius of the circumscribing sphere and  the cubic root of the ratio of total mass and the multiplier respectively. In this subsection, the mass at the top vertex of an $(n+1)$-body  pyramidal configuration is denoted by  $m_{n+1}$ and the position by $\q_{n+1}$.

We want to construct  co-spherical  central configurations that can extend  by  way IV and V of Theorem \ref{thm:class}. That is, we want the mass center to be  at the sphere center, or, $R_0=R$.

\subsubsection{Pyramidal central configurations}

Recall that an $(n+1)$-body  pyramidal central configuration is obtained by adding one arbitrary mass $m_{n+1}$ to a co-circular central configuration with the property $r_0>r$. The top vertex $\q_{n+1}$  is  on the orthogonal axis passing through the center of the circle, and the height   is $h= \sqrt{   r_0^2 -r^2}$.  Obviously, pyramidal configurations are co-spherical.

\begin{proposition}
	Let $\q$ be an $(n+1)$-body pyramidal central configuration.  Assume that the mass center of the co-circular base   is at the center of the circumscribing circle.  We  can choose $m_{n+1}$ such that the mass center of the pyramidal central configuration coincides with the circumscribing sphere center  if and only if $r_0>\sqrt{2}r$. 

\end{proposition}

		\begin{figure} [!h] \label{fig:triangle}
		\begin{center}
		\includegraphics [width=0.9 \textwidth] {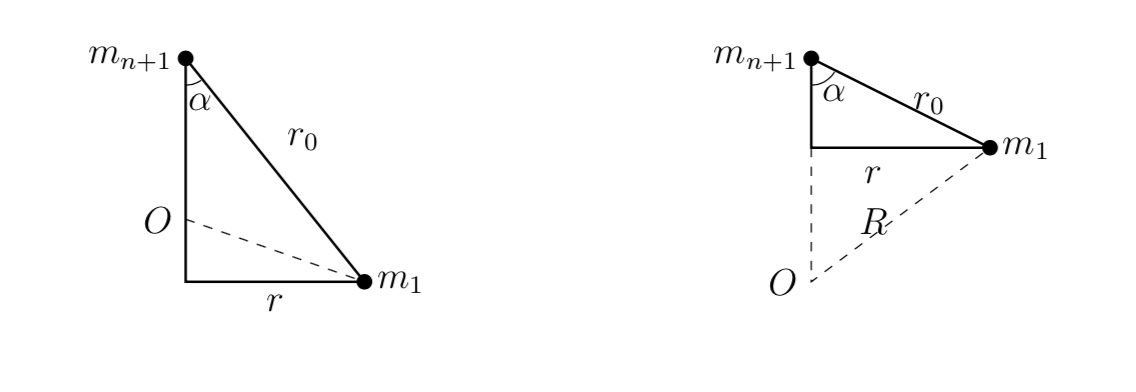}
	\end{center}
		\caption{ Slice of  a pyramidal central configuration. The centre of the circumscribing sphere is marked with $O$.   }
	\end{figure}

\begin{proof}
  Obviously, the sphere center $O$ is between the base and $m_{n+1}$   if and only if  $h= \sqrt{   r_0^2 -r^2}>r$, see Figure 3, left.
    That is, $r_0>\sqrt{2}r$.    The mass center is always between the base and $m_{n+1}$.  Thus, to make the mass center equals the sphere center, we must have $r_0>\sqrt{2}r$. 
  
  On the other hand, if $r_0>\sqrt{2}r$, it is easy to find $m_{n+1}$ such that the two centers  equal. 
	\end{proof}

\begin{proposition}
	Let $\q$ be an $(n+1)$-body pyramidal central configuration extended from an $n$-body co-circular central configuration.  Then $R_0=R$ if and only if $r_0=\frac{2}{\sqrt{3}} r$.
\end{proposition}

\begin{proof}
By Corollary \ref{cor:r_0}, we see that $R_0=r_0$.   	Note that $r_0 = R$ if and only if that $ \a= 60^\circ$, or, $\sin \a = \frac{r}{r_0}= \frac{\sqrt{3}}{2}$, see Figure 3, right. 
\end{proof}

\textbf{Examples}: 
Consider the  central configurations of  regular $n$-gon with equal masses.  Suppose that the $m_1=1$,  the radius of the circle  is $1$, and that  the  positions are $ \q_k=  e^{ \sqrt{-1} \th_k}, \  \   \th_k= \frac{2k\pi}{n}, \ k=1, ..., n. $    Recall that $\frac{r_0}{r}=(\frac{n}{ A(n)})^\frac{1}{3} $,  and it  is decreasing with respect to $n$.

Computation by Matlab shows 
\[  (\frac{8}{ A(8)})^\frac{1}{3}>\sqrt{2}, \  (\frac{9}{ A(9)})^\frac{1}{3}<\sqrt{2},   \  \  (\frac{52}{ A(52)})^\frac{1}{3}> \frac{2}{\sqrt{3}},  \ (\frac{53}{ A(53)})^\frac{1}{3}< \frac{2}{\sqrt{3}}.  \]   

We can draw two conclusions about the  $(n+1)$-body pyramidal central configurations extended from the regular $n$-gon central configurations 
$(n\le 472)$,  see Figure \ref{fig:s2cc},  left. 
\begin{enumerate}
	\item $R_0\ne R$;
	\item Only for $n=3, 4, 5, 6, 7, 8$, we can choose a proper top  mass  to make that the mass center of the pyramidal central configuration coincides with the sphere center. They can extend to $(n+2)$-body central configuration by adding one arbitrary mass at the center. As commented after Proposition \ref{prop:pcc}, the total collision solutions associated with them are   perverse solutions of the $(n+2)$-body problem \cite{Che01}. This 
	was noticed first by 	Ouyang-Xie-Zhang \cite{OXZ04}. 
\end{enumerate}
\subsubsection{Bi-Pyramidal central configurations}
By bi-pyramidal  configurations, we mean   configurations of $n+2$ bodies of which $n$ bodies  are co-planar and the other two being off the plane and in opposite directions. 
The regular $n$-gon with equal masses also generates  $(n+2)$-body bi-pyramidal  co-spherical central configurations.    Similar construction has been considered by Zhang-Zhou \cite{ZhZh01}. 

Place the  $n$-gon with equal masses on the equator, and two equal masses at the north and south pole, see Figure \ref{fig:s2cc}, right.  Assume that  the masses are $m_1=1, ...,m_n= 1, m_{n+1}=m_{n+2}=a$,  and the positions are 
 $ \q_k=( \cos \th_k, \sin \th_k, 0 ),  \ \th_k= \frac{2k\pi}{n},  k=1, ..., n, \ 
  \q_{n+1}=( 0,0,1 ),\     \q_{n+2}=( 0,0,-1 ).$ 
   \begin{figure}[!h]
		\begin{center}
		\includegraphics [width=1 \textwidth] {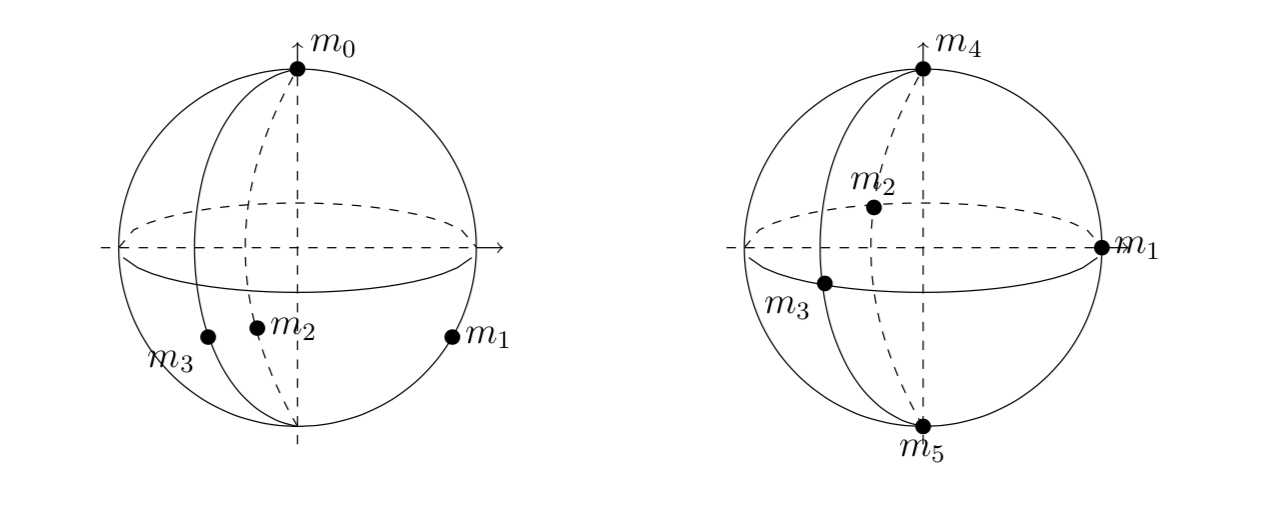}
	\end{center}
	\caption{Examples of co-spherical  central configuration generated from the equilateral triangle configuration.   }
	\label{fig:s2cc}  
\end{figure}

The  mass center is at the origin. The symmetry reduces the central configuration equations to the following system, 
\begin{align*}
& -\lambda \q_1=  \sum_{i=2}^n\frac{\q_k-\q_1}{r_{k1}^3} +   a  \frac{\q_{n+1} -\q_1}{r_{n+1,1}^3} +a  \frac{\q_{n+2}-\q_1}{r_{n+1,1}^3}=    -( A(n)+  \frac{a}{\sqrt{2}})\q_1,                                            \\
&  -\lambda \q_{n+1} =\sum_{i=1}^n\frac{\q_k-\q_{n+1}}{r_{n+1,k}^3} +    a\frac{\q_{n+2}-\q_{n+1}}{r_{n+1,n+2}^3}= -( \frac{n}{2\sqrt{2}}+\frac{a}{4}    )\q_{n+1}.   
\end{align*}
Here we use the fact that  $r_{n+1, 1}=...=r_{n+1, n}=\sqrt{2},  r_{n+1, n+2}=2$ and that  the sub configuration on the equator is central,  
\[ \sum_{i=2}^n\frac{\q_k-\q_1}{r_{k1}^3} =   - \frac{1}{r_0^3} (\sum_{i=1}^n m_i) \q_1=    - \frac{A(n)}{n} n \q_1.      \]
The system holds for positive $a=m_{n+1} $ if and only if $\frac{n} {A(n)}> 2\sqrt{2}.$ We have showed that this happens 
if and only if  $3\le n\le 8. $

\begin{proposition}
	The bi-pyramidal $(n+2)$-body configurations  constructed above are central with positive masses 
	if and only if $3\le n\le 8$. 
\end{proposition}

For all of them, the mass center equals the sphere center. Thus, they extend to $(n+3)$-body configurations by adding one arbitrary mass at the center. 
 Direct computation  shows  $R_0>R=1$ for all of them.

\section{Conclusions}

We have  classified the extensions of  
 $n$-body central configurations to    $(n+1)$-body central configurations in $R^3$.  For the collinear case,  the extensions happen only if  $n= 2$, so it is well understood.  
 For the non-collinear case, 
  the $n$-body  central configurations must be co-circular or co-spherical.    The co-circular (co-spherical) central configurations can extend if the mass center equals the geometric center, or $r_0\ge r$ ($r_0=r$ for the co-spherical case).  
  We also obtain a property on the value of $r_0$ for co-circular central configurations.  This  enables us to prove  the inequality $r_0>r$  for all  four, five and six-body  co-circular  central configurations. We solve the two questions of Hampton completely.  
  It might be worth noting that most of our proof remains valid for more general potentials and higher dimensional spaces.

There exist many  research works  on  
co-circular central configurations. 
We hope that this work may spark similar  interest to the co-spherical ones. 
The value $r_0=(\frac { m}{ \lambda})^{\frac{1}{3}}$ has showed its importance in the study of  four and five-body planar convex  central configurations \cite{CH18, MB32}.
 Our work reveals its another role 		
 in  the study  of central configurations.      Many  questions arise for the value $r_0$. For example, except from the trivial case $n=2$, do there exist co-circular or co-spherical central configurations with the property $r_0=r$? Can one obtain some general property of $r_0$ for the co-spherical  central configurations?
We hope to explore some of these questions in future work.

\section{acknowledgments}
Xiang Yu is supported by  NSFC(No.11701464) and the Fundamental Research Funds for the Central Universities (No. JBK1805001).  Shuqiang Zhu is supported by NSFC(No.11801537, No.11721101) and the Fundamental Research Funds for the Central Universities 
(No.WK0
010450010).


\begin{thebibliography}{FF}
	
	\bibitem{Alb03} A. Albouy, {\it  On a Paper of Moeckel on Central Configurations},   
	Regul. Chaotic Dyn. 8 (2003), no. 2, 133-142. 
	
	
	\bibitem{ACS12}
	A. Albouy, H.E. Cabral, A.A. Santos, 
	{\it Some Problems on the Classical N-Body Problem,} 
	Celestial Mech. Dynam. Astronom. 113 (2012), no. 4, 369-375. 

\bibitem{AK12} A. Albouy, V. Kaloshin,  \textit{Finiteness of Central Configurations of Five Bodies in the Plane}, Annals of Mathematics 176 (2012), no. 1, 535-588.
	
	\bibitem{ASV13} 	Martha Alvarez-Ram\'{i}rez,   Alan Almeida Santos, Claudio Vidal,
	{\it   On Co-Circular Central Configurations in the Four and Five
		Body-Problems for Homogeneous Force Law}, J. Dynam. Differential Equations, 25(2013), no.2, 269-290. 
	
	
	\bibitem{CH18}Kuo-Chang Chen,   Jun-Shian Hsiao,
	{\it Strictly Convex Central Configurations of the Planar Five-Body Problem}, Trans. Amer. Math. Soc. 370 (2018), no. 3, 1907-1924.

\bibitem{Che01} A. Chenciner,   {\it  Are There Perverse Choreographies?} New advances in celestial mechanics and Hamiltonian systems, 63-76, Kluwer/Plenum, New York, 2004. 




	\bibitem{CLP14} Montserrat  Corbera, Jaume Llibre,  Ernesto	 P\'{e}rez-Chavela, 	 {\it Spatial Bi-Stacked Central Configurations Formed by Two Dual
	Regular Polyhedra},  J. Math. Anal. Appl. 413(2014), no. 2, 648-659.

	
		\bibitem{CR12}Josep  M Cors,  Gareth E  Roberts, {\it  Four-Body Co-Circular Central Configurations},  Nonlinearity 25 (2012), no. 2, 343-370.  
		
	\bibitem{DS15} 	Chunhua Deng, Xia Su,  {\it Twisted Stacked Central Configurations for the Spatial
	Nine-Body Problem}, Z. Angew. Math. Phys. 66(2015), no. 4 1329-1339. 

\bibitem{Eul} Leonard Euler,  {\it  De Motu Rectilineo Trium Corporum se Mutuo Attrahentium,}  Novi Comm. Acad. Sci. Imp. Petrop. 11(1767), 144-151.	 

\bibitem{Fay96} N.  Fay\c{c}al,  {\it On the Classification of Pyramidal Central Configurations},  Proc. Amer. Math. Soc. 124 (1996), no. 1, 249-258.

\bibitem{FM13} A. C.  Fernandes,   L. F.  Mello,  {\it  
	On Stacked Planar Central Configurations with Five Bodies When One Body is Removed},
Qual. Theory Dyn. Syst. 12 (2013), no. 2, 293-303. 


\bibitem{FM13-1} A. C.  Fernandes,   L. F.  Mello,  {\it  
	On Stacked Planar Central Configurations with $n$ Bodies When One Body is Removed},
J. Math. Anal. Appl. 405(2013), no.1, 320-325. 


\bibitem{FM15}A. C.  Fernandes,  L. F.  Mello,  {\it  Rigidity of Planar Central Configurations,}   Z. Angew. Math. Phys. 66 (2015), no. 6, 2979-2994.

\bibitem{FM18} A. C.  Fernandes,  L. F.  Mello, {\it  Correction to: On Stacked Planar Central Configurations with Five Bodies When One Body Is Removed}, 
Qual. Theory Dyn. Syst. (2018).  https://doi.org/10.1007/s12346-018-0280-5. 


\bibitem{Ham03} M.  Hampton, 
{\it   Co-Circular Central Configurations in the Four-Body Problem,} EQUADIFF 2003, 993-998, World Sci. Publ., Hackensack, NJ, 2005.

\bibitem{Ham05} M.  Hampton, 
{\it   Stacked Central Configurations: New Examples in the Planar Five-Body Problem},  
Nonlinearity 18 (2005), no. 5, 2299-2304.

	\bibitem{HS07} Marshall Hampton, Manuele Santoprete,  \textit{ 
	Seven-Body Central Configurations: a Family of Central Configurations in the Spatial Seven-Body Problem}, Celestial Mech. Dynam. Astronom. 99 (2007), no. 4, 293-305.  

\bibitem{Lag} J.L. Lagrange, {\it Essai sur le Probl\`eme des Trois Corps}, 1772, \OE{}uvres tome 6, 229-332. 

\bibitem{LMS15}   J. Llibre, R. Moeckel,  C. Sim\'{o}, \textit{Central Configurations, Periodic Orbits, and Hamiltonian Systems}, Advanced Courses in Mathematics,  CRM Barcelona, Birkh\"{a}user/Springer, Basel, 2015.




\bibitem{Moe90}R. Moeckel, \emph{On Central Configurations},  
Math. Z. 205 (1990), no. 4, 499-517. 

\bibitem{MB32}W. D. MacMillan and Walter Bartky,  {\it Permanent Configurations in the Problem of Four Bodies, } 
Trans. Amer. Math. Soc. 34 (1932), no. 4, 838-875.







\bibitem{MS95}  R. Moeckel,  Carles Sim\'o,  {\it  Bifurcation of Spatial Central Configurations From Planar Ones,}  
SIAM J. Math. Anal. 26 (1995), no. 4, 978-998. 


	\bibitem{OC12} 	 Allyson  Oliveira, Hildeberto Cabral, 	 {\it  On Stacked Central Configurations of the Planar Coorbital
	Satellites Problem},  Discrete Contin. Dyn. Syst. 32(2012), no.10, 3715-3732. 

\bibitem{OXZ04} T. 
Ouyang, Z. Xie,  S.  Zhang, {\it Pyramidal Central Configurations and Perverse Solutions,}     Electron. J. Differential Equations,  2004, No. 106, 9 pp.



\bibitem{Saa80} D.\ Saari, {\it On the Role and Properties of Central Configurations},  Celestial Mech. 21 (1980), 9-20.

	
	\bibitem{Sma70-2} S. Smale, \textit{Topology and Mechanics, II. The Planar $N$-Body problem},  Invent. Math. 11 (1970), 45–64.
	
	
\bibitem{ZhZh01}S. Zhang, Q. Zhou,  {\it Double Pyramidal Central Configurations},    Phys. Lett. A 281 (2001), no. 4, 240-248.
	
	
	
	


	
 
	


	
	


\end{thebibliography}
\end{document}